\documentclass[11pt]{article}

\usepackage{tikz}
\usetikzlibrary{shapes}
\usetikzlibrary{decorations.pathreplacing}
\usepackage{enumerate}
\usepackage[shortlabels]{enumitem}
\usepackage{verbatim}
\usepackage[margin=1in]{geometry}
\usepackage{amssymb}
\usepackage{caption}
\usepackage[normalem]{ulem}
\usepackage{bbm}
\usepackage{amsthm}
\usepackage{enumitem}
\usepackage{scalerel}    
\usepackage{stmaryrd}   

\usetikzlibrary{arrows.meta}

\usepackage[nottoc,notlot,notlof]{tocbibind}

\DeclareFontFamily{U}{mathx}{\hyphenchar\font45}
\DeclareFontShape{U}{mathx}{m}{n}{
      <5> <6> <7> <8> <9> <10>
      <10.95> <12> <14.4> <17.28> <20.74> <24.88>
      mathx10
      }{}
\DeclareSymbolFont{mathx}{U}{mathx}{m}{n}
\DeclareFontSubstitution{U}{mathx}{m}{n}
\DeclareMathAccent{\widecheck}{0}{mathx}{"71}
\DeclareMathAccent{\wideparen}{0}{mathx}{"75}

\usepackage{rotating,centernot,cancel}    

\newcommand{\myfootnote}[1]{
    \renewcommand{\thefootnote}{}
    \footnotetext{\scriptsize#1}
    \renewcommand{\thefootnote}{\arabic{footnote}}
}

\usepackage{amsmath}
\usepackage{tocloft}
\usepackage{float}

\usepackage{tcolorbox}   
\PassOptionsToPackage{hyphens}{url}
\usepackage[hidelinks]{hyperref}
\usepackage{babel}
\theoremstyle{plain}

\input amssym.def
\input amssym.tex

\def\beqn{\begin{eqnarray}}
\def\eeqn{\end{eqnarray}}

\newcommand{\olsi}[1]{\underline{{#1}}}

\def\beq{\begin{equation}}
\def\eeq{\end{equation}}
\usepackage{mathtools}
\DeclarePairedDelimiter\floor{\lfloor}{\rfloor}

\newtheorem{theorem}{Theorem}[section]

\newtheorem{lemma}[theorem]{Lemma} 
\newtheorem*{lemma*}{Lemma}
\newtheorem{proposition}[theorem]{Proposition} 

\theoremstyle{remark}

\theoremstyle{definition}

\numberwithin{figure}{section}
\numberwithin{equation}{section}

\def\Z{\mathbb Z}

\def\dlim[#1][#2]{\lim_{#1 \to #2, #1 \neq #2}}

\def\Var{\textup{$\mathbb{V}$ar}}

\def\Cov{\textup{$\mathbb{C}$ov}}

\newcommand{\be}{\begin{equation}}
\newcommand{\ee}{\end{equation}}

\def\wt{\widetilde}

\newcounter{cnstcnt}

\usepackage{xargs}
\usepackage[colorinlistoftodos,prependcaption,textsize=footnotesize]{todonotes}
\newcommandx{\addmath}[2][1=]{\todo[linecolor=red,backgroundcolor=red!25,bordercolor=red,#1]{#2}}
\newcommandx{\fixtext}[2][1=]{\todo[linecolor=blue,backgroundcolor=blue!25,bordercolor=blue,#1]{#2}}
\newcommandx{\note}[2][1=]{\todo[linecolor=yellow,backgroundcolor=yellow!25,bordercolor=yellow,#1]{#2}}

\newcommand\bbullet{{{\scaleobj{0.6}{\bullet}}}} 

\title{Time correlations in the inverse-gamma polymer\\ with flat initial condition}
\author{
 Xiao Shen\thanks{\scriptsize{Department of Mathematics, University of Utah, Utah, USA. \texttt{xiao.shen@utah.edu}}}
  }
\date{}
\begin{document}

\allowdisplaybreaks

\maketitle

\begin{abstract}

Temporal correlations in the KPZ universality class have gained significant attention, following the conjectures in \cite{Ferr-Spoh-2016}. Building on prior work in the zero temperature setting \cite{timecorrflat}, we address the time correlation problem with flat initial conditions in the positive temperature regime. Our study focuses on the inverse-gamma polymer, where we establish an upper bound for the correlation between two free energies whose endpoints are far apart in time. In contrast to the previous work \cite{timecorrflat}, our work not only extends the result to positive temperatures but also eliminates the reliance on integrable probability inputs related to the Airy process.  This advancement allows us to address local scales, where the short time remains fixed while the large time grows arbitrarily, a scenario beyond the reach of the Airy scaling limit.

\end{abstract}

\myfootnote{Date: \today}
\myfootnote{2020 Mathematics Subject Classification. 60K35, 	60K37}
\myfootnote{Key words and phrases: correlation, coupling, directed polymer, Kardar-Parisi-Zhang, random growth model.}


\section{Introduction}
The study of universality has long been a central theme in probability theory. A classic example is the central limit theorem (CLT), which asserts that under mild moment assumptions, the large-scale behavior of the sum of independent and identically distributed (i.i.d.)\ random variables is universal, converging to the Gaussian distribution after appropriate scaling. In this sense, the random growth characterized by the sum of i.i.d.\ random variables can be considered a member of the \textit{Gaussian universality class}.

In 1986, Kardar, Parisi, and Zhang introduced a distinct universality class now known as the \textit{KPZ universality class} \cite{Kar-Par-Zha-86}. This class includes models with complex spatial correlations and is very rich, encompassing percolation models, directed polymers, interacting particle systems, certain stochastic partial differential equations, and more. Unlike the Gaussian universality class, where the scaling limit converges to the Gaussian distribution, the KPZ class is expected to have a scaling limit that converges to various Tracy-Widom distributions from  Random Matrix Theory.

Proving that a given model belongs to the KPZ universality class has been a formidable challenge over the past 38 years. This process is reminiscent of the progress made in proving the CLT, which evolved from Bernoulli distributions in the 1700s to general distributions in the early 1900s. Only a few exactly solvable models with specific weight distributions have been rigorously verified to belong to the KPZ class, similar to the role of Bernoulli distribution in proving the CLT. Establishing KPZ universality for models with general weight distributions remains a great challenge.

The remarkable structure of exactly solvable models facilitates the use of algebraic tools, enabling the derivation of explicit formulas for one-point and multi-point distributions of growth profiles in these models. This approach, particularly successful as recognized in \cite{Bai-Dei-Joh-99}, is known as \textit{integrable probability}. While integrable probability is a powerful tool, it has a drawback: adapting this approach to general models remains unclear, as the methods rely heavily on specific formulas of weight distributions. With the eventual goal of  extending results beyond exactly solvable cases, researchers have developed several alternative approaches for these exactly solvable models based on broadly applicable probabilistic techniques and geometric arguments.

A notable technique in this area involves using black-box integrable probability inputs combined with probabilistic and geometric arguments. These integrable inputs help extract detailed information about the geodesic geometry, leading to a better understanding of the space-time profile. This approach was significantly advanced by Basu, Sidoravicius, and Sly, who used it to solve the celebrated slow bond conjecture \cite{slowbondproblem}.

Another important method involves coupling two random growth processes: one with a specific initial condition and another with a stationary initial condition, which is often easier to analyze. By controlling the differences between these two systems, one can draw conclusions about the growth process with the specific initial condition of interest. This technique was first introduced in the context of KPZ models by Sepp\"al\"ainen in the 1990s through a series of works on interacting particle systems. It was later employed in the influential work \cite{Cat-Gro-06} to obtain KPZ exponents. More recently, this approach was further refined in \cite{rtail1}, leading to the development of quantitatively optimal bounds.

As alluded to above, there has been significant interest in revisiting old results and removing the reliance on integrable probability techniques when analyzing exactly solvable models \cite{balzs2019nonexistence, cgm_low_up, rtail1, opt_exit, Lan-Sos-22-a-, OC_tail,coalnew, seppcoal}. In line with this, we investigate the time correlation in the inverse-gamma polymer model with flat initial conditions. Our study builds on previous work \cite{timecorrflat}, which explored this time correlation problem in the zero temperature setting, namely,  the exponential last-passage percolation. Here, we eliminate the need for integrable inputs related to the Airy process in the upper bound analysis. Additionally, the extension to positive temperatures requires different and additional estimates, which we will discuss after presenting our main result (Theorem \ref{thm1}) in the next section.

\subsection{Main result}

While the spatial statistics of the height function in exactly solvable models are fairly well understood, temporal correlations remain more challenging to characterize. Following the experimental and numerical studies by physicists \cite{Take-2012, Take-Sano-2012}, the precise conjectures were formulated by Ferrari and Spohn in \cite{Ferr-Spoh-2016}. Since then, a series of works have addressed these temporal correlation problems across different models and initial conditions \cite{timecorriid, timecorrflat, bas-sep-she-24, diff_timecorr, KPZcorr, ferocctimecorr, fer-occ-2022, Ferr-Spoh-2016}.  We will not delve into a detailed discussion of the existing literature here; instead, we refer interested readers to the introduction of \cite{bas-sep-she-24} for a comprehensive overview.

We work with a representative positive temperature model known as the inverse-gamma polymer. To define this model, 
recall that a random variable $X$ has the inverse-gamma distribution with shape parameter $\mu\in(0,\infty)$, denoted as $X\sim \text{Ga}^{-1}(\mu)$, if $X^{-1}$ has the gamma distribution with the shape parameter $\mu$.
To define the inverse-gamma polymer on $\mathbb{Z}^2$, let $\{Y_{\bf z}\}_{\mathbf z\in \mathbb{Z}^2}$ be i.i.d.~inverse-gamma distributed random variables with a fixed shape parameter $\mu \in (0, \infty)$. For two coordinatewise-ordered vertices ${\bf u}$ and ${\bf v}$ of $\mathbb{Z}^2$, let $\mathbb{X}_{{\bf u}, {\bf v}}$ denote the collection of up-right paths with unit steps between them, and the \textit{point-to-point partition function} (which excludes the weight at the end point) is defined by 
$$
Z_{\mathbf u, \mathbf v} = \frac{1}{Y_{\mathbf v}}\sum_{\boldsymbol\gamma \in \mathbb{X}_{\mathbf u, \mathbf v}} \prod_{\mathbf{z}\in \boldsymbol\gamma} Y_{\mathbf z},
$$
We use the convention $Z_{\mathbf u,\mathbf v} =  0$ if $\mathbf u\leq \mathbf v$ fails. 
Moreover, the \textit{point-to-point free energy} is defined as $\log Z_{\mathbf u, \mathbf v}$.

To define the polymer model with the flat initial condition (or the line-to-point model), let us denote the anti-diagonal line through the origin by $\mathcal{L}_{\mathbf 0}$. For any vertex $\mathbf v$ above $\mathcal{L}_{\mathbf 0}$, we define the \textit{line-to-point partition function} as follows
$$Z_{\mathcal{L}_{\mathbf 0},\mathbf{v}} = \sum_{k \in \mathbb{Z}} Z_{(k, -k), \mathbf v}.$$
And similarly, the \textit{line-to-point free energy} will  be $\log Z_{\mathcal{L}_{\mathbf 0},\mathbf{v}}$. 
With these definitions, we are ready to state the result of our paper.
\begin{theorem}\label{thm1}
There exist positive constants $C_1, c_0, n_0$ such that, whenever $n\geq n_0$ and $c_0 \le r \leq n/2$, it holds that
$$ \textup{$\mathbb{C}$ov}\Big(\log Z_{\mathcal{L}_{\mathbf 0}, (r,r)}, \log Z_{\mathcal{L}_{\mathbf 0}, (n,n)}\Big) \leq C_1\frac{r^{4/3}}{n^{2/3}} \log^{100}(n/r).$$ 
\end{theorem}

The upper bound is expected to be of order $r^{4/3}/n^{2/3}$, with the logarithmic expression arising as an error term from our analysis. Compared to the zero temperature case in \cite{timecorrflat}, which employs geometric arguments with integrable probability inputs, our proof avoids relying on such inputs related to the Airy process. This refinement enables us to extend the upper bound to local scales, broadening the range from $\delta n \leq r \leq n$ in \cite{timecorrflat} to $c_0 \leq r \leq n$. Furthermore, we have slightly reduced the error term in the upper bound, improving it from $\exp(C\log^{5/6}(n/r))$ in \cite{timecorrflat} to $\log^{100}(n/r)$.

We also note that transitioning from zero temperature to positive temperature results in several of our estimates that differ from those in \cite{timecorrflat}. Specifically, in the zero temperature setting, one estimate used throughout the upper bound proof in \cite{timecorrflat} is that the difference between the last-passage value and the last-passage value restricted to a wide parallelogram is zero with high probability. This estimate is particularly useful when the difference term appears within a product of several other terms, quickly leading to an upper bound of zero for the product. In the positive temperature setting, however, this approach no longer applies, necessitating detailed tail estimates for all the terms and the development of new arguments.

Lastly, regarding the lower bound of the covariance in Theorem \ref{thm1}, an optimal lower bound of order \( r^{4/3}/n^{2/3} \) was established in \cite{timecorrflat} for the exponential last-passage percolation. With the recent convergence of the inverse-gamma free energy profile to the Airy process \cite{ag_airy}, we anticipate that a similar result could be obtained for the inverse-gamma polymer. However, replacing the dependence on the Airy process and extending these results to cover the local scale will be part of our future work. Additionally, for the case when both $n$ and $r$ are large, i.e., $n/2 \leq r \leq n-c_0$, both upper and lower bounds have been established in Theorem 1.1 of \cite{diff_timecorr}.

\vspace{3mm}
\noindent\textbf{Acknowledgements.} 
The author sincerely thanks Riddhipratim Basu for insightful discussions regarding his work in \cite{timecorrflat}. The author acknowledges partial support from the Wylie Research Fund at the University of Utah.

\section{Preliminaries}

\subsection{Notation}\label{not}
Fix ${\bf u}\in \mathbb{Z}^2$, let us denote the anti-diagonal line through $\mathbf{u}$ as
$\mathcal{L}_{{\bf u}}=\{{\bf u}+(j,-j): j\in\Z\}$.
For $k\in \mathbb{R}_{\geq 0}$, define
$\mathcal{L}^k_{{\bf u}}$ to be the line segment $\{{\bf x}\in \mathcal{L}_{{\bf u}} : |{\bf x}-{\bf u}|_\infty \leq k  \}.$
For two coordinatewise-ordered  points ${\bf u} \leq {\bf v}$ and $k\in \mathbb{R}_{\geq 0}$,  
$R_{{\bf u}, {\bf v}}^k$ denotes the parallelogram spanned by the four corners ${\bf u} \pm (k,-k)$ and ${\bf v} \pm (k,-k)$.

Integer points on the diagonal or the anti-diagonal are abbreviated as $a=(a,a)$ and $\olsi{a} = (a,-a)$ when they occur as subscripts. Common occurrences of this include  $\mathcal{L}_{(r,r)} = \mathcal{L}_r$, $Z_{\mathcal{L}_{(r,r)}, (n,n)} = Z_{\mathcal{L}_r, n}$,  $Z_{{\bf p}, (N+k,N-k)} = Z_{{\bf p}, N+\olsi{k}}$,  and $R_{(a,a), (b,b)}^k = R_{a, b}^k$.

For the polymer partition functions, in our paper, we will use the following:
\begin{align*}
Z_{A, B}^{\textup{in}, R^{h}_{{\bf c}, {\bf d}}} &= \textup{ the partition function with paths from $A$ to $B$ contained inside $R^{h}_{{\bf c}, {\bf d}}$}\\ 
Z_{A, B}^{\textup{out}, R^{h}_{{\bf c}, {\bf d}}} &= \textup{ the partition function with paths from $A$ to $B$ contained outside $R^{h}_{{\bf c}, {\bf d}}$}\\ 
Z_{A, B}^{\textup{exit}, R^{h}_{{\bf c}, {\bf d}}} &= \textup{ the partition function with paths from $A$ to $B$  that exit diagonal sides of  $R^{h}_{{\bf c}, {\bf d}}$}\\
Z_{A, B}^{\textup{touch}, R^{h}_{{\bf c}, {\bf d}}} &= \textup{ the partition function with paths from $A$ to $B$  that intersect $R^{h}_{{\bf c}, {\bf d}}$}
\end{align*}
To simplify notation, when the subscripts of the partition function match those of the parallelogram appearing in the superscript, we use the abbreviation \( Z_{\mathcal{L}_{{\bf a}}^{s_1}, \mathcal{L}_{{\bf b}}^{s_2}}^{\star, k} \) for \( Z_{\mathcal{L}_{{\bf a}}^{s_1}, \mathcal{L}_{{\bf b}}^{s_2}}^{\star, R^{k}_{{\bf a}, {\bf b}}} \), where \(\star\) represents ``in," ``out," ``exit," or ``touch."

We adopt two conventions for clarity regarding constants and integer rounding. First, generic positive constants will be denoted by \(C, C'\), etc., throughout the calculations and proofs, with the understanding that these constants may vary from line to line without a change in notation. Second, we simplify expressions by omitting the integer floor function. For instance, if the line segment from \((0,0)\) to \((N, N)\) is divided into 5 equal parts, we denote the free energy of the first segment by \(\log Z_{0, N/5}\), even if \(N/5\) is not an integer.

\subsection{Estimates for the polymer model}\label{est_poly_bulk}

In this section, we will restate several estimates from \cite{bas-sep-she-24, diff_timecorr}. 
Recall the \textit{shape function} which represents the law of large numbers limit of the point-to-point free energy. It is a deterministic continuous function \(\Lambda: \mathbb{R}^2_{\geq 0} \rightarrow \mathbb{R}\) that satisfies

\begin{equation}\label{shape}
\lim_{n\rightarrow \infty} \sup_{ |{\bf z}|_1 \geq n} \frac{|\log {Z}_{0, {\bf z}} - \Lambda({\bf z})|}{|{\bf z}|_1} = 0 \qquad \textup{$\mathbb{P}$-almost surely}.
\end{equation}
For simplicity, let us abbreviate $\Lambda\big((n,n)\big)$ as $\Lambda_n$.

Our first two propositions give the upper bounds for the right and left tails of the free energy.
\begin{proposition}\label{ptl_upper}
There exist positive constants $C_1, C_2, n_0$ such that for each $n \geq n_0$, $t\geq 1$ and $1\leq h \leq e^{C_1{\min\{t^{3/2}, tn^{1/3}\}}}$,
$$\mathbb{P}\Big(\log{Z}_{\mathcal{L}_0^{hn^{2/3}},\mathcal{L}_{n}}  - \Lambda_n \geq tn^{1/3}\Big) \leq 
e^{-C_2\min\{t^{3/2}, tn^{1/3}\}}.  
$$
\end{proposition}

\begin{proposition}\label{low_ub}
There exist positive constants $C_1, n_0 $ such that for each $n\geq n_0$, $t\geq 1$ , we have 
$$\mathbb{P}(\log Z_{0,  n} - \Lambda_n \leq -tn^{1/3}) \leq e^{-C_1 \min\{t^{3/2}, tn^{1/3}\}}.$$
\end{proposition}
A direct consequence of the two tail estimates above is the following variance bound for the point-to-line free energy. 
\begin{proposition}\label{varbd}
There exist positive constants $C_1, n_0$ such that for each $n\geq n_0$,  we have
$$\Var\Big(\log Z_{\mathcal{L}_0, n}\Big) \leq Cn^{2/3}.$$
\end{proposition}

The next two propositions summarize the loss of free energy for paths with too much transversal fluctuation. While the original estimates do not include the $s$ parameter below, the same proof can be adapted, and we omit the details. The constant $c^*$ will appear later in our proofs, so we introduce this special notation.

\begin{proposition}\label{trans_fluc_loss}
There exist positive constants $c^*, C_1, n_0$ such that for each $n\geq n_0$, $h \geq 0$, $t\geq 1$ and $s\geq 0$ we have
$$\mathbb{P}\Big(\log Z_{\mathcal{L}^{hn^{2/3}}_0, \mathcal{L}_{n}\setminus\mathcal{L}_{n}^{(h+t)n^{2/3}}}    - \Lambda_n \geq (-c^* t^2 + s)n^{1/3} \Big) \leq e^{-C_1(t+s)}.$$
\end{proposition}

The next proposition generalizes the loss of free energy to paths with high transversal fluctuation somewhere along their entire length.  Recall \( Z_{\mathcal{L}_{{\bf a}}^{s_1}, \mathcal{L}_{{\bf b}}^{s_2}}^{\textup{exit}, k} \) is the partition function that sums over directed paths from  $\mathcal{L}_{{\bf a}}^{s_1}$  to  $\mathcal{L}_{{\bf b}}^{s_2}$ that exit diagonal sides of $R^{k}_{{\bf a}, {\bf b}}$. 
\begin{proposition}\label{trans_fluc_loss3}
There exist positive constants $c^*, C_1, n_0$ such that for each $n\geq n_0$, $s \geq 0$, $1 \leq t \leq  n^{1/3}$ and $0<h < e^{t}$,  we have
$$\mathbb{P}\Big(\log {Z}^{\textup{exit}, {(h+t)n^{2/3}}}_{\mathcal{L}_{0}^{hn^{2/3}}, \mathcal{L}_n^{hn^{2/3}}} - \Lambda_N \geq (-c^* t^2+s)n^{1/3}\Big) \leq e^{-C_1(t+s)}.$$
\end{proposition}

Lastly, the proposition below shows when we constrain our paths to a parallelogram that obeys the KPZ scale, the free energy will not be too small.  Recall \( Z_{\mathcal{L}_{{\bf a}}^{s_1}, \mathcal{L}_{{\bf b}}^{s_2}}^{\textup{in}, k} \) is the partition function that sums over directed paths from  $\mathcal{L}_{{\bf a}}^{s_1}$  to  $\mathcal{L}_{{\bf b}}^{s_2}$ that are contained  inside of $R^{k}_{{\bf a}, {\bf b}}$. 

\begin{proposition}\label{wide_similar}
There exist positive constants $C_1, t_0$ such that for each $0< \theta \leq  100$, there exists a positive constant $n_0$ such that for each $n\geq n_0$, $t\geq t_0$ and  ${\bf p} \in \mathcal{L}_n^{\theta n^{2/3}}$, we have
$$\mathbb{P}\Big(\log {Z}^{\textup{in}, {\theta n^{2/3}}}_{0, {\bf p}} - \Lambda_n \leq -tn^{1/3}\Big) \leq  \tfrac{\sqrt{t}}{\theta}e^{-C\theta t}.$$
\end{proposition}

\section{Proof of the theorem }

\begin{figure}[t]
\begin{center}

\tikzset{every picture/.style={line width=0.75pt}} 

\begin{tikzpicture}[x=0.75pt,y=0.75pt,yscale=-1,xscale=1]

\draw [color={rgb, 255:red, 155; green, 155; blue, 155 }  ,draw opacity=1 ]   (100.46,110.54) -- (208.78,219.94) ;
\draw [color={rgb, 255:red, 155; green, 155; blue, 155 }  ,draw opacity=1 ]   (130.4,80.48) -- (238.72,189.88) ;
\draw  [fill={rgb, 255:red, 0; green, 0; blue, 0 }  ,fill opacity=1 ] (177.84,129.72) .. controls (177.84,128.68) and (178.68,127.84) .. (179.72,127.84) .. controls (180.76,127.84) and (181.6,128.68) .. (181.6,129.72) .. controls (181.6,130.76) and (180.76,131.6) .. (179.72,131.6) .. controls (178.68,131.6) and (177.84,130.76) .. (177.84,129.72) -- cycle ;
\draw  [fill={rgb, 255:red, 0; green, 0; blue, 0 }  ,fill opacity=1 ] (238.44,70.4) .. controls (238.44,69.36) and (239.28,68.52) .. (240.32,68.52) .. controls (241.36,68.52) and (242.2,69.36) .. (242.2,70.4) .. controls (242.2,71.44) and (241.36,72.28) .. (240.32,72.28) .. controls (239.28,72.28) and (238.44,71.44) .. (238.44,70.4) -- cycle ;
\draw    (140.97,152.06) .. controls (163.29,130.02) and (164.4,143.77) .. (179.72,129.72) ;
\draw    (106.4,116.91) .. controls (174.8,72.99) and (162.04,121.56) .. (240.32,70.4) ;
\draw  [dash pattern={on 0.84pt off 2.51pt}]  (135.14,145.33) -- (165.46,115.89) ;
\draw  [dash pattern={on 0.84pt off 2.51pt}]  (164.69,174.7) -- (195.01,145.26) ;
\draw    (141.33,91.62) .. controls (157.25,87.38) and (151.38,95.56) .. (177.82,96.22) ;
\draw [color={rgb, 255:red, 155; green, 155; blue, 155 }  ,draw opacity=1 ]   (301.79,110.54) -- (410.11,219.94) ;
\draw [color={rgb, 255:red, 155; green, 155; blue, 155 }  ,draw opacity=1 ]   (331.73,80.48) -- (440.05,189.88) ;
\draw  [fill={rgb, 255:red, 0; green, 0; blue, 0 }  ,fill opacity=1 ] (379.17,129.72) .. controls (379.17,128.68) and (380.02,127.84) .. (381.05,127.84) .. controls (382.09,127.84) and (382.93,128.68) .. (382.93,129.72) .. controls (382.93,130.76) and (382.09,131.6) .. (381.05,131.6) .. controls (380.02,131.6) and (379.17,130.76) .. (379.17,129.72) -- cycle ;
\draw  [fill={rgb, 255:red, 0; green, 0; blue, 0 }  ,fill opacity=1 ] (439.77,70.4) .. controls (439.77,69.36) and (440.62,68.52) .. (441.65,68.52) .. controls (442.69,68.52) and (443.53,69.36) .. (443.53,70.4) .. controls (443.53,71.44) and (442.69,72.28) .. (441.65,72.28) .. controls (440.62,72.28) and (439.77,71.44) .. (439.77,70.4) -- cycle ;
\draw    (342.3,152.06) .. controls (364.62,130.02) and (365.73,143.77) .. (381.05,129.72) ;
\draw    (353.2,142.8) .. controls (421.6,98.87) and (363.38,121.56) .. (441.65,70.4) ;
\draw  [dash pattern={on 0.84pt off 2.51pt}]  (336.48,143.73) -- (366.8,114.29) ;
\draw  [dash pattern={on 0.84pt off 2.51pt}]  (364.82,174.7) -- (395.14,145.26) ;
\draw  [fill={rgb, 255:red, 0; green, 0; blue, 0 }  ,fill opacity=1 ] (139.65,89.93) -- (143.02,89.93) -- (143.02,93.3) -- (139.65,93.3) -- cycle ;
\draw    (131.79,151.84) -- (157.81,178.96) ;
\draw [shift={(159.2,180.4)}, rotate = 226.17] [color={rgb, 255:red, 0; green, 0; blue, 0 }  ][line width=0.75]    (10.93,-3.29) .. controls (6.95,-1.4) and (3.31,-0.3) .. (0,0) .. controls (3.31,0.3) and (6.95,1.4) .. (10.93,3.29)   ;
\draw [shift={(130.4,150.4)}, rotate = 46.17] [color={rgb, 255:red, 0; green, 0; blue, 0 }  ][line width=0.75]    (10.93,-3.29) .. controls (6.95,-1.4) and (3.31,-0.3) .. (0,0) .. controls (3.31,0.3) and (6.95,1.4) .. (10.93,3.29)   ;
\draw    (332.19,152.24) -- (358.21,179.36) ;
\draw [shift={(359.6,180.8)}, rotate = 226.17] [color={rgb, 255:red, 0; green, 0; blue, 0 }  ][line width=0.75]    (10.93,-3.29) .. controls (6.95,-1.4) and (3.31,-0.3) .. (0,0) .. controls (3.31,0.3) and (6.95,1.4) .. (10.93,3.29)   ;
\draw [shift={(330.8,150.8)}, rotate = 46.17] [color={rgb, 255:red, 0; green, 0; blue, 0 }  ][line width=0.75]    (10.93,-3.29) .. controls (6.95,-1.4) and (3.31,-0.3) .. (0,0) .. controls (3.31,0.3) and (6.95,1.4) .. (10.93,3.29)   ;

\draw (386.36,117.84) node [anchor=north west][inner sep=0.75pt]    {$( r,r)$};
\draw (240.96,52.04) node [anchor=north west][inner sep=0.75pt]    {$( n,n)$};
\draw (135.36,66.04) node [anchor=north west][inner sep=0.75pt]    {$u_{r}^{\textup{max}}$};
\draw (118.8,167.44) node [anchor=north west][inner sep=0.75pt]    {$r^{2/3}$};
\draw (318.4,167.84) node [anchor=north west][inner sep=0.75pt]    {$r^{2/3}$};
\draw (444.96,51.04) node [anchor=north west][inner sep=0.75pt]    {$( n,n)$};

\end{tikzpicture}

\captionsetup{width=0.8\textwidth}
\caption{The heuristic suggests that with high probability close to \(1 - (r/n)^{2/3}\), the ``optimal paths" for the two free energies $\log Z_{\mathcal{L}_0, r}$ and $\log Z_{\mathcal{L}_0, n}$ are disjoint (illustrated on the left), resulting in no contribution to the covariance  $\textup{$\mathbb{C}$ov}\big(\log Z_{\mathcal{L}_{0}, r}, \log Z_{\mathcal{L}_{0}, n}\big)$. 
However, in the rare event with a probability of order \((r/n)^{2/3}\), the ``optimal paths" will significantly overlap (illustrated on the right). This overlap leads to the desired order of covariance \((r/n)^{2/3} \cdot r^{2/3}\), where the term \(r^{2/3}\) comes from the variance of $\log Z_{\mathcal{L}_0, r}$. Finally, in the analysis, part of the free energy $\log Z_{\mathcal{L}_0, n}$ is approximated by the line-to-point free energy $\log Z_{\mathcal{L}_r, n}$, with the maximizer labeled as $u^{\textup{max}}_r$ on the left. This maximizer is expected to be of order \(n^{2/3}\) away from the diagonal.}  \label{fig1}
\end{center}
\end{figure}
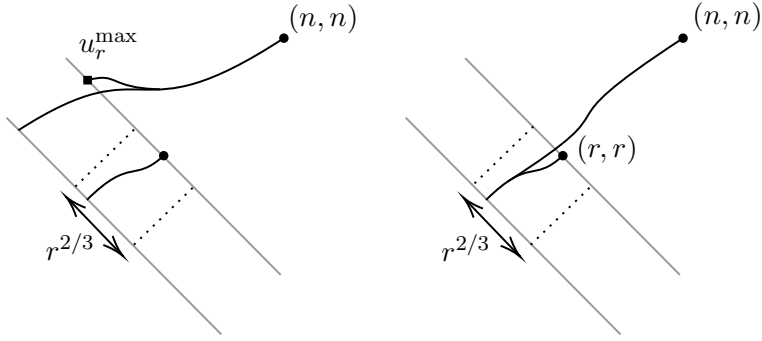
We follow the geometric approach outlined in the zero temperature work in \cite{timecorrflat}, with heuristics tracing back to \cite{Ferr-Spoh-2016}. 
However, several of our estimates will differ from those in \cite{timecorrflat} due to the lack of geodesics. 

First, to start the proof, we may assume that $r/n \leq \ell_0$ for some absolute constant $0<\ell_0<1$, because otherwise by Cauchy-Schwartz inequality, 
\begin{align*}\textup{$\mathbb{C}$ov}\Big(\log Z_{\mathcal{L}_{\mathbf 0}, (r,r)}, \log Z_{\mathcal{L}_{\mathbf 0}, (n,n)}\Big) &\leq \Var\Big(\log Z_{\mathcal{L}_{\mathbf 0}, (r,r)}\Big)^{1/2} \Var\Big(\log Z_{\mathcal{L}_{\mathbf 0}, (n,n)}\Big)^{1/2}\\
& \leq Cr^{1/3}n^{1/3} \qquad \text{by Proposition \ref{varbd}}\\
&\leq C\ell_0^{-1} r^{2/3} \leq C\ell_0^{-2} \frac{r^{4/3}}{n^{2/3}}.
\end{align*} 
Thus, in our proof, we may assume that the value $\ell_0$ is sufficiently small whenever needed. 

Next, let $\mathbf{u}^{\textup{max}}_r$ denote the unique maximizer of 
$$\max \Big\{\log Z_{\mathbf{u}, n} : \mathbf u \in \mathcal{L}_r \Big\}.$$
Fix $j_0$ such that $10^9j_0 r^{2/3} = n$.
For $j=1,2, \dots, j_0$, let us define the event
$$A_j = \Big\{\Big|\mathbf{u}^{\textup{max}}_r\cdot \mathbf e_1 - \mathbf{u}^{\textup{max}}_r \cdot \mathbf e_2\Big|  \in \Big[10^9 (j-1)r^{2/3}, 10^9 jr^{2/3}\Big)\Big\}.$$
Let $\wt j = 1\vee \floor{\log^{10} j}$, and note that $\wt j + 100 \leq 10^9 j$ for $j = 1, \dots, j_0$. Now, define $B_j \subset A_j$ to be the event that in addition, the following inequality holds
$$ \log Z_{\mathbf{u}^{\textup{max}}_r, n} - \max_{\mathbf{u} \in \mathcal{L}_r^{5 \wt j r^{2/3}}} \log Z_{\mathbf{u},n}  \geq  {\wt j}^{1/2} r^{1/3}.$$
Finally, let $C_j = A_j \setminus B_j$.

Let  $\mathcal{F}_r$ to denote the $\sigma$-algebra generated by weights $Y_\mathbf{z}$ for $\mathbf z$ that lie on or above the anti-diagonal line $\mathcal{L}_r$. By independence, we have that 
\begin{align*}\Cov\Big(\log Z_{\mathcal{L}_0, r}, \log Z_{\mathcal{L}_0, n}\Big) &= \mathbb{E}\Big[\log Z_{\mathcal{L}_0, r} \log Z_{\mathcal{L}_0, n}\Big] - \mathbb{E}\Big[\log Z_{\mathcal{L}_0, r}\Big]\mathbb{E}\Big[ \log Z_{\mathcal{L}_0, n}\Big]\\
& = \mathbb{E}\Big[\mathbb{E}\Big[\log Z_{\mathcal{L}_0, r} \log Z_{\mathcal{L}_0, n} - \mathbb{E}\big[\log Z_{\mathcal{L}_0, r}\big|\mathcal{F}_r\big]\mathbb{E}\big[ \log Z_{\mathcal{L}_0, n}\big|\mathcal{F}_r\big]\Big|\mathcal{F}_r\Big]\Big]\\
& =\mathbb{E}\Big[\Cov_{\mathcal{F}_r}(\log Z_{\mathcal{L}_0, r}, \log Z_{\mathcal{L}_0, n})\Big].
\end{align*}
where $\Cov_{\mathcal{F}_r}$ denotes the conditional covariance with respect to $\mathcal{F}_r$.

We now decompose the expectation above according to the disjoint events $B_j$ and $C_j$, in particular, we note that all $B_j$ and $C_j$ are $\mathcal{F}_r$-measurable. Now, to conclude the proof,  it suffices for us to bound
\begin{align}
\mathbb{E}[\Cov_{\mathcal{F}_r}(\log Z_{\mathcal{L}_0, r}, \log Z_{\mathcal{L}_0, n})] & =\sum_{j=1}^{j_0} \mathbb{E} \Big[\mathbbm{1}_{B_j} \Cov_{\mathcal{F}_r}(\log Z_{\mathcal{L}_0, r}, \log Z_{\mathcal{L}_0, n})\Big]\label{Bevent}\\
&\qquad \qquad + \sum_{j=1}^{j_0} \mathbb{E} \Big[\mathbbm{1}_{C_j} \Cov_{\mathcal{F}_r}(\log Z_{\mathcal{L}_0, r}, \log Z_{\mathcal{L}_0, n})\Big]\label{Cevent}
\end{align}

First, we will look at the expectation inside the sum from \eqref{Bevent}. Recall the restricted free energies $\log Z_{\mathcal{L}_0, \bbullet}^{\textup{in}, R_{0,r}^{\wt j r^{2/3}}}$ and $\log Z_{\mathcal{L}_0, \bbullet}^{\textup{out}, R_{0,r}^{\wt j r^{2/3}}}$, which are defined in Section \ref{not}.
We have the following equality
\begin{align}&\mathbb{E}\Big[\mathbbm{1}_{B_j} \Cov_{\mathcal{F}_r}(\log Z_{\mathcal{L}_0, r}, \log Z_{\mathcal{L}_0, n})\Big]\nonumber\\
& =\mathbb{E}\Big[\mathbbm{1}_{B_j} \Cov_{\mathcal{F}_r}\Big(\log Z_{\mathcal{L}_0, r} - \log Z_{\mathcal{L}_0, r}^{\textup{in}, \wt j r^{2/3}}, \log Z_{\mathcal{L}_0, n}- Z_{\mathcal{L}_0, n}^{\textup{out}, R_{0,r}^{\wt j r^{2/3}}}\Big)\Big]\label{D1}\\
& \qquad + \mathbb{E}\Big[\mathbbm{1}_{B_j} \Cov_{\mathcal{F}_r}\Big( \log Z_{\mathcal{L}_0, r}^{\textup{in}, \wt j r^{2/3}}, \log Z_{\mathcal{L}_0, n}- Z_{\mathcal{L}_0, n}^{\textup{out}, R_{0,r}^{\wt j r^{2/3}}}\Big)\Big]\label{D2}\\
&\qquad \qquad+\mathbb{E}\Big[\mathbbm{1}_{B_j} \Cov_{\mathcal{F}_r}\Big(\log Z_{\mathcal{L}_0, r} - \log Z_{\mathcal{L}_0, r}^{\textup{in}, \wt j r^{2/3}}, \log Z^{\textup{out}, R_{0,r}^{\wt j r^{2/3}}}_{\mathcal{L}_0, n}\Big)\Big]\label{D3}
\end{align}
We will be bounding each of the three lines above using several estimates, and these estimates will be proven in the next section.

We start with \eqref{D1} and \eqref{D2}, which uses the following result. 
\begin{proposition}\label{prop_e} There exist positive constant $C_1, C_2, r_0$ such that for each $r\geq r_0$ and $j=1,\dots, j_0$, the following estimates hold:
\begin{enumerate}[(i)]
\item $\Var\Big(\log Z_{\mathcal{L}_0, r}^{\textup{in}, \wt j r^{2/3}}\Big)^{1/2} \leq C_1r^{1/3}$\label{e1}
\item $\mathbb{E}\Big[\Big(\log Z_{\mathcal{L}_0, r} - \log Z_{\mathcal{L}_0, r}^{\textup{in}, \wt j r^{2/3}}\Big)^{10}\Big]\leq e^{-C_2\wt j}r^{10/3}$\label{e2}
\item $\mathbb{E}\Big[\mathbbm{1}_{B_j} \mathbb{E}\Big[\Big( \log Z_{\mathcal{L}_0, n}- Z_{\mathcal{L}_0, n}^{\textup{out}, R_{0,r}^{\wt j r^{2/3}}}\Big)^2\,\Big|\,\mathcal{F}_r \Big]^{1/2}\Big] \leq C_1e^{-C_2\wt j} \frac{r}{n^{2/3}} \log^{100}(n/r)$\label{e3}
\end{enumerate}
\end{proposition}
Then, by Cauchy-Schwartz inequality, the fact that $\Var(X) \leq \mathbb{E}[X^2]$, and the random variable $\log Z_{\mathcal{L}_0, r} - \log Z_{\mathcal{L}_0, r}^{\textup{in}, \wt j r^{2/3}}$ is independent of $\mathcal{F}_r$, we obtain
\begin{align*}
\eqref{D1} &\leq \mathbb{E}\Big[\Big(\log Z_{\mathcal{L}_0, r} - \log Z_{\mathcal{L}_0, r}^{\textup{in}, \wt j r^{2/3}}\Big)^2\Big]^{1/2}  \cdot \mathbb{E}\Big[\mathbbm{1}_{B_j} \mathbb{E}\Big[\Big( \log Z_{\mathcal{L}_0, n}- Z_{\mathcal{L}_0, n}^{\textup{out}, R_{0,r}^{\wt j r^{2/3}}}\Big)^2\,\Big|\,\mathcal{F}_r \Big]^{1/2}\Big]\\
&\leq C_1 e^{-C_2 \wt j}\frac{r^{4/3}}{n^{2/3}}\log^{100}(n/r) \qquad \textup{ by Proposition \ref{prop_e} \ref{e2} and \ref{e3}}\\
\eqref{D2} &\leq \Var\Big(\log Z_{\mathcal{L}_0, r}^{\textup{in}, \wt j r^{2/3}}\Big)^{1/2}  \cdot \mathbb{E}\Big[\mathbbm{1}_{B_j} \mathbb{E}\Big[\Big( \log Z_{\mathcal{L}_0, n}- Z_{\mathcal{L}_0, n}^{\textup{out}, R_{0,r}^{\wt j r^{2/3}}}\Big)^2\,\Big|\,\mathcal{F}_r \Big]^{1/2}\Big]\\
&\leq C_1 e^{-C_2 \wt j}\frac{r^{4/3}}{n^{2/3}}\log^{100}(n/r) \qquad \textup{ by Proposition \ref{prop_e} \ref{e1} and \ref{e3}}
\end{align*}
The estimate for \eqref{D3} is different. Note that because $\log Z_{\mathcal{L}_0, r} - \log Z_{\mathcal{L}_0, r}^{\textup{in}, \wt j r^{2/3}} \geq 0$ by definition, then 
\begin{align}
\eqref{D3} &\leq  \mathbb{E}\Big[\mathbbm{1}_{B_j} \Big(\log Z_{\mathcal{L}_0, r} - \log Z_{\mathcal{L}_0, r}^{\textup{in}, \wt j r^{2/3}}\Big)\Big|\log Z^{\textup{out}, R_{0,r}^{\wt j r^{2/3}}}_{\mathcal{L}_0, n} - \mathbb{E}\Big[\log Z^{\textup{out}, R_{0,r}^{\wt j r^{2/3}}}_{\mathcal{L}_0, n}\,\Big|\,\mathcal{F}_r\Big]\Big|\Big]\nonumber\\
&\leq  \int_{[0, \infty)} \mathbb{P}\Big(\mathbbm{1}_{B_j} \Big(\log Z_{\mathcal{L}_0, r} - \log Z_{\mathcal{L}_0, r}^{\textup{in}, \wt j r^{2/3}}\Big)\Big|\log Z^{\textup{out}, R_{0,r}^{\wt j r^{2/3}}}_{\mathcal{L}_0, n} - \mathbb{E}\Big[\log Z^{\textup{out}, R_{0,r}^{\wt j r^{2/3}}}_{\mathcal{L}_0, n}\,\Big|\,\mathcal{F}_r\Big]\Big| \geq x\Big) dx \nonumber\\
&\leq r^{2/3}\log^2(n/r) \int_{[0, \infty)} \mathbb{P}\Big(\mathbbm{1}_{B_j} \Big(\log Z_{\mathcal{L}_0, r} - \log Z_{\mathcal{L}_0, r}^{\textup{in}, \wt j r^{2/3}}\Big)\nonumber\\
& \qquad \qquad \qquad \qquad \qquad \qquad \Big|\log Z^{\textup{out}, R_{0,r}^{\wt j r^{2/3}}}_{\mathcal{L}_0, n} - \mathbb{E}\Big[\log Z^{\textup{out}, R_{0,r}^{\wt j r^{2/3}}}_{\mathcal{L}_0, n}\,\Big|\,\mathcal{F}_r\Big]\Big| \geq tr^{2/3}\log^2(n/r)\Big) dt \nonumber\\
&\leq r^{2/3}\log^2(n/r) \int_{[0, \infty)} \mathbb{P}\Big(B_j \cap \Big\{\log Z_{\mathcal{L}_0, r} - \log Z_{\mathcal{L}_0, r}^{\textup{in}, \wt j r^{2/3}} \geq \wt j^{-10}\sqrt{t}r^{1/3}\log(n/r)\Big\}\Big) \label{term1}\\
& \qquad   + \mathbb{E}\Big[ \mathbbm{1}_{B_j}\mathbb{P}\Big(\Big|\log Z^{\textup{out}, R_{0,r}^{\wt j r^{2/3}}}_{\mathcal{L}_0, n} - \mathbb{E}\Big[\log Z^{\textup{out}, R_{0,r}^{\wt j r^{2/3}}}_{\mathcal{L}_0, n}\,\Big|\,\mathcal{F}_r\Big]\Big| \geq \wt j^{10}\sqrt{t}r^{1/3}\log(n/r)\Big) \,\Big|\, \mathcal{F}_r\Big)\Big] dt . \label{term2}
\end{align}
Now we will bound the probability and the expectation in \eqref{term1} and \eqref{term2} respectively. In the process, we will also make use of the following results. 
\begin{proposition}\label{bj}There exist positive constants $C_1, r_0$ such that for each $r\geq r_0$ and $j=1,\dots, j_0$, the following estimate holds
$$\mathbb{P}(B_j) \leq C_1j (r/n)^{2/3} \log^{20}(n/r).$$
\end{proposition}
\begin{proposition}\label{prop_A}
There exist positive constants $r_0, t_0$ such that for each $r\geq r_0$ and $t\geq t_0$, there exists an $\mathcal{F}_r$-measurable event $G_t$ with $\mathbb{P}(G_t)\geq 1-e^{-\sqrt t}$ and the following holds,
$$\mathbbm{1}_{G_t\cap B_j}\mathbb{P}\Big(\Big|\log Z_{\mathcal{L}_0, n}^{\textup{out}, R_{0,r}^{\wt j r^{2/3}}}  - \mathbb{E}\Big[\log Z_{\mathcal{L}_0, n}^{\textup{out}, R_{0,r}^{\wt j r^{2/3}}}  \,\Big| \,\mathcal{F}_r \Big] \Big| \geq tr^{1/3}\log(n/r) \, \Big| \, \mathcal{F}_r \Big) \leq e^{-\sqrt t}.$$
\end{proposition}

Coming back to the estimates, note that by independence, Markov inequality and Proposition \ref{prop_e} \ref{e2}
\begin{align*}
\textup{the probability in } \eqref{term1} & = \mathbb{P}(B_j) \mathbb{P}\Big(\log Z_{\mathcal{L}_0, r} - \log Z_{\mathcal{L}_0, r}^{\textup{in}, \wt j r^{2/3}} \geq \wt j^{-10}\sqrt{t}r^{1/3}\log(n/r)\Big) \\
& \leq \min\Big\{\mathbb{P}(B_j) , \mathbb{P}(B_j) \frac{\mathbb{E}\Big[\Big(\log Z_{\mathcal{L}_0, r} - \log Z_{\mathcal{L}_0, r}^{\textup{in}, \wt j r^{2/3}}\Big)^{10}\Big]}{(\wt j^{-10}\sqrt{t}r^{1/3}\log(n/r))^{10}}\Big\}\\
& \leq \min\Big\{\mathbb{P}(B_j) , \mathbb{P}(B_j)  \frac{C_1e^{-C_2\wt j}\wt j^{100}}{t^5}\Big\}.
\end{align*}

By setting the variable $t$ in Proposition \ref{prop_A} to be $\wt j^{10} \sqrt{t}$, we have  
\begin{align*}
\textup{the expectation in } \eqref{term2} & \leq \mathbb{P}(B_j\cap G_{\wt j^{10}\sqrt{t}})e^{-\wt jt^{1/5}} + \mathbb{P}(B_j \cap G_{\wt j^{10}\sqrt{t}}^c). 
\end{align*}
Now to continue the calculation with \eqref{term1} and \eqref{term2} from before, it holds that 
\begin{align*}
\textup{\eqref{term1} and \eqref{term2} }&\leq C_1 r^{2/3}\log^2(n/r) \int_{[0, \infty)} \min\Big\{\mathbb{P}(B_j) , \mathbb{P}(B_j)\frac{\wt j^{100}e^{-C_2\wt j}}{ t^{5}}\Big\} \\
&\qquad \qquad \qquad \qquad \qquad \qquad + \mathbb{P}(B_j\cap G_{\wt j^{10}\sqrt t})e^{-\wt jt^{1/5}} + \mathbb{P}(B_j \cap G_{\wt j^{10}\sqrt t}^c) dt \\
&\leq C_1 r^{2/3}\log^2(n/r) \Big(\wt j^{100} e^{-C_2\wt j} \mathbb{P}(B_j) + \int_{[0, \infty)} \mathbb{P}(B_j) e^{-\wt jt^{1/5}}dt\\
& \qquad \qquad \qquad \qquad \qquad \qquad + \int_{[0, \log^{10}\mathbb{P}(B_j))}\mathbb{P}(B_j)d t + \int_{[\log^{10}\mathbb{P}(B_j), \infty)} e^{-\wt jt^{1/5}}d t\Big)\\
& \leq C_1 e^{-C_2\wt j} \wt j^{200} j \frac{r^{4/3}}{n^{2/3}}\log^{30}(n/r)\qquad \textup{by Proposition \ref{bj}}.
\end{align*}
With this, we have completed the estimates for \eqref{D1}, \eqref{D2}, and \eqref{D3}. Recall $\wt j = 1\vee \floor{\log^{10}j}$, then it holds that
$$\eqref{Bevent} \leq \sum_{j=1}^{j_0}C_1  \wt j^{200} j  e^{-C_2 \wt j }\frac{r^{4/3}}{n^{2/3}}\log^{100}(n/r) \leq C\frac{r^{4/3}}{n^{2/3}}\log^{100}(n/r).$$

Next, we turn to the estimate for \eqref{Cevent}. Again, we will look at the expectation term inside the sum, and we have that 
\begin{align}
\mathbb{E} \Big[\mathbbm{1}_{C_j} &\Cov_{\mathcal{F}_r}(\log Z_{\mathcal{L}_0, r}, \log Z_{\mathcal{L}_0, n})\Big]\nonumber\\
&\leq  \mathbb{E} \Big[\mathbbm{1}_{C_j} |\log Z_{\mathcal{L}_0, r} - \mathbb{E}[\log Z_{\mathcal{L}_0, r}]| \Big|\log Z_{\mathcal{L}_0, n} - \mathbb{E}\Big[\log Z_{\mathcal{L}_0, n}\,\Big|\,\mathcal{F}_r\Big]\Big|\Big]\nonumber\\
&\leq  \int_{[0, \infty)} \mathbb{P}\Big(\mathbbm{1}_{C_j} |\log Z_{\mathcal{L}_0, r} - \mathbb{E}[\log Z_{\mathcal{L}_0, r}]| \Big|\log Z_{\mathcal{L}_0, n} - \mathbb{E}\Big[\log Z_{\mathcal{L}_0, n}\,\Big|\,\mathcal{F}_r\Big]\Big| \geq x\Big) dx \nonumber\\
&\leq r^{2/3}\log^2(n/r) \int_{[0, \infty)} \mathbb{P}\Big(\mathbbm{1}_{C_j} |\log Z_{\mathcal{L}_0, r} - \mathbb{E}[\log Z_{\mathcal{L}_0, r}]|\nonumber\\
& \qquad \qquad \qquad \qquad \qquad \qquad  \Big|\log Z_{\mathcal{L}_0, n} - \mathbb{E}\Big[\log Z_{\mathcal{L}_0, n}\,\Big|\,\mathcal{F}_r\Big]\Big|  \geq tr^{2/3}\log^2(n/r)\Big) dt \nonumber\\
&\leq r^{2/3}\log^2(n/r) \int_{[0, \infty)} \mathbb{P}\Big(C_j \cap \Big\{|\log Z_{\mathcal{L}_0, r} - \mathbb{E}[\log Z_{\mathcal{L}_0, r}]| \geq\sqrt{t}r^{1/3}\log(n/r)\Big\}\Big) \label{term11}\\
& \qquad   + \mathbb{E}\Big[ \mathbbm{1}_{C_j}\mathbb{P}\Big(\Big|\log Z_{\mathcal{L}_0, n} - \mathbb{E}\Big[\log Z_{\mathcal{L}_0, n}\,\Big|\,\mathcal{F}_r\Big]\Big| \geq \sqrt{t}r^{1/3}\log(n/r)\Big) \,\Big|\, \mathcal{F}_r\Big)\Big] dt. \label{term22}
\end{align}
Note that the probability in \eqref{term11} is upper bounded by $\mathbb{P}(C_j)e^{-C \sqrt t}$ by independence,  Proposition \ref{ptl_upper} and Proposition \ref{low_ub}. And the expectation in \eqref{term22} is similar to \eqref{term2}, which utilizes the following proposition. 
\begin{proposition}\label{prop_AA}
There exist constants $r_0, t_0$ such that for each $r\geq r_0$ and $t\geq t_0$, there exists an $\mathcal{F}_r$-measurable event $G_t$ with $\mathbb{P}(G_t)\geq 1-e^{-\sqrt t}$ and the following holds,
$$\mathbbm{1}_{G_t}\mathbb{P}\Big(\Big|\log Z_{\mathcal{L}_0, n} - \mathbb{E}\Big[\log Z_{\mathcal{L}_0, n}  \,\Big| \,\mathcal{F}_r \Big] \Big| \geq tr^{1/3}\log(n/r) \, \Big| \, \mathcal{F}_r \Big) \leq e^{-\sqrt t}.$$
\end{proposition}
Then, the expectation in \eqref{term22} is upper bounded by $\mathbb{P}(C_j \cap G_{\sqrt t})e^{-t^{1/5}} + \mathbb{P}(C_j \cap G_{\sqrt t}^c)$.
Combining these and continue the estimates for \eqref{term11} and \eqref{term22}, 
\begin{align*}
\textup{\eqref{term11} and \eqref{term22} }&\leq r^{2/3}\log^2(n/r) \int_{[0, \infty)} \mathbb{P}(C_j)e^{-C \sqrt t}  + \mathbb{P}(C_j \cap G_{\sqrt t})e^{-t^{1/5}} + \mathbb{P}(C_j \cap G_{\sqrt t}^c) dt \\
&\leq r^{2/3}\log^2(n/r) \Big(C\mathbb{P}(C_j)  + \int_{[0, \log^{10}\mathbb{P}(C_j))}\mathbb{P}(C_j)d t + \int_{[\log^{10}\mathbb{P}(C_j), \infty)} e^{-t^{1/5}}d t\Big)\\
& \leq  Cr^{2/3}\log^2(n/r)\mathbb{P}(C_j)\log^{10}\mathbb{P}(C_j).
\end{align*}

The next proposition gives the estimate needed to continue, and we recall that the events $C_j$ are disjoint by definition.
\begin{proposition}\label{cj}
The following estimate hold:
$\sum_{j=1}^{j_0} \mathbb{P}(C_j) \leq C\frac{r^{2/3}}{n^{2/3}} \log^{20}(n/r)$
\end{proposition}

Then, it holds that 
$$\eqref{Cevent} \leq \sum_{j=1}^{j_0}Cr^{2/3}\log^2(n/r)\mathbb{P}(C_j)\log^{10}\mathbb{P}(C_j) \leq Cr^{2/3}\log^{20}(n/r) \sum_{j=1}^{j_0}\mathbb{P}(C_j) \leq C\frac{r^{4/3}}{n^{2/3}}\log^{100}(n/r). $$
Now, we have shown the desired estimates for \eqref{Bevent} and \eqref{Cevent} and completed the proof of the theorem.

\section{Proof of the estimates}

In this section, we prove the propositions from the previous section.

\subsection{Proof of Proposition \ref{bj} and Proposition \ref{cj}}

To remove the calculation from the Airy process, we will utilize estimate (5.25) from \cite{buse_ind} which essentially states that for  $n^{-2/3} < \delta \leq \delta_0$ and $ |m| \leq |\log \delta|n^{2/3}$, 
it holds that 
\begin{equation}\label{upest}\mathbb{P}\Big(\max_{|i| \leq {2|\log \delta| n^{2/3}}}\Big[\log Z_{0, n+\olsi{i}} \Big] - \max_{|j| \leq {\delta n^{2/3}}} \Big[\log Z_{0,n+ \olsi{m+j})} \Big]  \leq |\log \delta|^2 \sqrt{\delta}n^{1/3}\Big) \leq C|\log \delta|^{10} \delta.
\end{equation}

We will start by considering Proposition \ref{cj}, and recall that \( j_0 = \frac{n}{10^9r^{2/3}} \). We will divide the argument into two cases based on whether \( \left(\frac{n}{r}\right)^2 \) is less than or greater than or equal to \( j_0 \). 
We begin with the first case. First, consider \( j = \left(\frac{n}{r}\right)^2, \dots, j_0 \). Note that in this case, \( jr^{2/3} \geq \left(\frac{n}{r}\right)^2 r^{2/3} = \left(\frac{n}{r}\right)^{4/3} n^{2/3} \), which is large on the time scale \( n^{2/3} \). We can then proceed with the following estimates: 
\begin{align}
\sum_{j= (n/r)^2}^{j_0} \mathbb{P}(C_j) &\leq \mathbb{P}\Big(\bigcup_{j=(n/r)^2}^{j_0} A_j\Big) \nonumber \\
&\leq  \mathbb{P}\Big(\max_{\mathbf{u} \in \mathcal{L}_r\setminus \mathcal{L}_r^{(n/r)^2 r^{2/3}}} \log Z_{\mathbf u, n}= \max_{\mathbf{u} \in \mathcal{L}_r} Z_{\mathbf u, n}\Big)\nonumber\\
& \leq \mathbb{P}\Big(\max_{\mathbf{u} \in \mathcal{L}_r\setminus \mathcal{L}_r^{(n/r)^2 r^{2/3}}} \log Z_{\mathbf u, n}- \Lambda_{n-r} \geq -c^*(n/r)n^{1/3} \Big) \nonumber\\
& \qquad \qquad \qquad  + \Big(\max_{\mathbf{u} \in \mathcal{L}_r} \log Z_{\mathbf u, n}- \Lambda_{n-r} \leq -c^*(n/r)n^{1/3}\Big)\nonumber\\
&\leq e^{-C(n/r)} \leq r/n \qquad \textup{ by Proposition \ref{trans_fluc_loss} and Proposition \ref{low_ub}}.\nonumber
\end{align}
Now, for $j = 1, \dots, (n/r)^2$, the union $\bigcup_{j=1}^{(n/r)^2} C_j$ is contained inside the following event 
$$\Big\{\log Z_{\mathbf{u}^{\textup{max}}_r, n} - \max_{\mathbf{u} \in \mathcal{L}_r^{100\log^{10}(n/r) r^{2/3}}} \log Z_{\mathbf{u},n}  <   10\log^{5}(n/r) r^{1/3}\Big\}.$$
This can be upper bounded 
via \eqref{upest} by setting $m=0$ and $\delta=100\log^{10}(n/r)(r/n)^{2/3} \leq \delta_0$ provided that $r/n$ is fixed sufficiently small. 

In the second case when $j_0 \leq (n/r)^2$, the union $\bigcup_{j=1}^{j_0} C_j$ is contained inside the following event 
$$\Big\{\log Z_{\mathbf{u}^{\textup{max}}_r, n} - \max_{\mathbf{u} \in \mathcal{L}_r^{100\log^{10}j_0 r^{2/3}}} \log Z_{\mathbf{u},n}  <   10\log^5j_0 r^{1/3}\Big\}.$$
This can be again upper bounded 
via \eqref{upest} by setting $m=0$ and $\delta=100\log^5j_0(r/n)^{2/3}  \leq 200\log^5(n/r)(r/n)^{2/3} \leq \delta_0$ provided that $r/n$ is fixed sufficiently small. 
With this, we have finished the proof of Proposition \ref{cj}.

Turn to  Proposition \ref{bj}, note that when $j_0 \geq (n/r)^{2/3}$, then for $j= (n/r)^{2/3}, \dots, j_0$, $\mathbb{P}(B_j)$ can be upper bounded trivially as we have $$Cj(r/n)^{2/3}\log^{20}(n/r) \geq C\log^{20}(n/r) \geq 1 \geq \mathbb{P}(B_j)$$
provided that $n/r$ is fixed sufficiently large. 
Then, for the values of $j = 1, \dots, (n/r)^{2/3}\wedge j_0$, let $\delta = (r/n)^{2/3}$, and we have the following estimate 
\begin{align*}
\mathbb{P}(B_j) &\leq \mathbb{P}(A_j)\\
& \leq \sum_{|m|=100(j-1)}^{100j} \mathbb{P}\Big(\max_{|i| \leq {2|\log \delta| n^{2/3}}}\Big[\log Z_{0, n+\olsi{i}} \Big] = \max_{|j| \leq {\delta n^{2/3}}} \Big[\log Z_{0,n+ \olsi{mr^{2/3}+j})} \Big] \Big) \\
&\leq Cj(r/n)^{2/3}\log^{20}(n/r) \qquad \text{by \eqref{upest}.}
\end{align*}
With this, we finish the proof of Proposition \ref{bj}.

\subsection{Proof of Proposition \ref{prop_e}}
To start, note that the free energy $\log Z_{\mathcal{L}_0, r}^{\textup{in}, \wt j r^{2/3}}$ in \ref{e1} can be bounded from above and below as $$\log Z_{0, r}^{\textup{in},  r^{2/3}} \leq \log Z_{\mathcal{L}_0, r}^{\textup{in}, \wt j r^{2/3}} \leq \log Z_{\mathcal{L}_0, r}.$$
Then, from the right and left tail bounds from Proposition \ref{ptl_upper} and Proposition \ref{wide_similar}, we have $\Big|\mathbb{E}\Big[\log Z_{\mathcal{L}_0, r}^{\textup{in}, \wt j r^{2/3}}\Big] - \Lambda_r\Big| \leq Cr^{1/3}$ and $\mathbb{E}\Big[\Big(\log Z_{\mathcal{L}_0, r}^{\textup{in}, \wt j r^{2/3}} - \Lambda_r\Big)^2\Big] \leq Cr^{2/3}$. This directly implies the variance bound in Proposition \ref{prop_e} \ref{e1}.

For \ref{e2}, we will recall the estimate (4.7) from the arXiv version of \cite{diff_timecorr}, which states that for some constants $C_1$ and $C_2$,
\begin{equation}\label{bdj} \mathbb{P}\Big(  \log Z_{\mathcal{L}_0, r} - \log Z_{\mathcal{L}_0, r}^{\textup{in}, \wt j r^{2/3}}  \geq e^{-C_1\wt j^2 r^{1/3}}\Big) \leq e^{-C_2 \wt j^3}.
\end{equation}

Now, let us define the event $A = \Big\{\log Z_{\mathcal{L}_0, r} - \log Z_{\mathcal{L}_0, r}^{\textup{in}, \wt j r^{2/3}}  \geq e^{-C_1\wt j^2 r^{1/3}}\Big\}$, and the following estimate holds
\begin{align*}
&\mathbb{E}\Big[\Big(\log Z_{\mathcal{L}_0, r} - \log Z_{\mathcal{L}_0, r}^{\textup{in}, \wt j r^{2/3}}\Big)^{10}\Big]\\ 
&= \mathbb{E}\Big[\Big(\log Z_{\mathcal{L}_0, r} - \log Z_{\mathcal{L}_0, r}^{\textup{in}, \wt j r^{2/3}}\Big)^{10}\mathbbm{1}_{A^c}\Big] +  \mathbb{E}\Big[\Big(\log Z_{\mathcal{L}_0, r} - \log Z_{\mathcal{L}_0, r}^{\textup{in}, \wt j r^{2/3}}\Big)^{10}\mathbbm{1}_{A}\Big]\\
& \leq e^{-C\wt j} + \mathbb{E}\Big[\Big(\log Z_{\mathcal{L}_0, r} - \log Z_{\mathcal{L}_0, r}^{\textup{in}, \wt j r^{2/3}}\Big)^{20}\Big]^{1/2}\mathbb{P}(A)^{1/2}\\
& \leq e^{-C\wt j} + Cr^{10/3}e^{-C'\wt j}
\end{align*}
where the expectation bound in the last inequality follows again from Proposition \ref{ptl_upper} and Proposition \ref{wide_similar}, while the probability bound follows from \eqref{bdj}. With this, we complete the proof for Proposition \ref{prop_e} \ref{e2}.

In the remainder of this section, we will prove the following proposition, which, when combined with Proposition \ref{bj}, implies Proposition \ref{prop_e} \ref{e3}. 

\begin{proposition}
There exist positive constants $C_1$ and \(t_0\) such that for each \(j = 1, \dots, j_0\) and \(t \geq t_0\), let \(\wt j = 1 \vee \lfloor \log^{10}j \rfloor\) and the following inequality holds almost surely
$$\mathbbm{1}_{B_j} \mathbb{P}\Big( \log Z_{\mathcal{L}_0, n}- Z_{\mathcal{L}_0, n}^{\textup{out}, R_{0,r}^{\wt j r^{2/3}}} > t r^{1/3}\,\Big|\,\mathcal{F}_r \Big) \leq e^{-C_1({\wt j}^{1/2} + t)}.$$
\end{proposition}

\begin{proof} To start, note that by definition $\mathbb{P}(B_1) = 0$, thus for the rest of the proof, we will assume that $j \geq 2$. 

Recall $\log Z^{\textup{touch}, R_{0,r}^{\wt j r^{2/3}}}_{\mathcal{L}_0, n}$ to denote the free energy between $\mathcal{L}_0$ and $(n,n)$ with paths that all intersect the parallelogram $R_{0,r}^{\wt jr^{2/3}}$, and by definition, 
$$\log Z_{\mathcal{L}_0, n} \leq \max \Big\{\log Z^{\textup{touch}, R_{0,r}^{\wt j r^{2/3}}}_{\mathcal{L}_0, n} , \log  Z_{\mathcal{L}_0, n}^{\textup{out}, R_{0,r}^{\wt j r^{2/3}}}\Big \} + \log 2. $$
To prove the estimate in our proposition, it suffices for us to show the following  estimate 
$$\mathbbm{1}_{B_j} \mathbb{P}\Big( \log Z^{\textup{touch}, R_{0,r}^{\wt j r^{2/3}}}_{\mathcal{L}_0, n}- \log Z_{\mathcal{L}_0, n}^{\textup{out}, R_{0,r}^{\wt j r^{2/3}}} > t r^{1/3}\,\Big|\,\mathcal{F}_r \Big) \leq  e^{-C(\wt j^{1/2} + t)}.$$

To start, we could replace the term $\log Z^{\textup{touch}, R_{0,r}^{\wt j r^{2/3}}}_{\mathcal{L}_0, n}$ from our estimate  above by 
\begin{equation}
\max_{\mathbf{p} \in \mathcal{L}_r}  \Big\{\log Z^{\textup{touch}, R_{0,r}^{\wt j r^{2/3}}}_{\mathcal{L}_0, \mathbf p} + \log Z_{\mathbf p, n}\Big\}.\label{max1}
\end{equation}
because  
$\eqref{max1} \leq \log Z^{\textup{touch}, R_{0,r}^{\wt  j r^{2/3}}}_{\mathcal{L}_0, n} \leq \eqref{max1} + 100 \log r$ as the directed paths from $\mathbf{p}$ must touch the parallelogram $R_{0,r}^{\wt j r^{2/3}}$.
Let $\mathbf p^*$ denote the unique maximizer of \eqref{max1}. We will separately work on two disjoint events,
$
H = \big\{\mathbf{p}^* \in \mathcal{L}_r^{4\wt 
 jr^{2/3}}\big\} $ and
$H^c$.

Starting with the event $H$, in our estimate below, we use the fact that for $j \geq 2$, if $\mathbf {v}\in \mathcal{L}_r$ satisfies $|\mathbf v \cdot \mathbf e_1 -\mathbf v\cdot \mathbf e_2| \in [10^9(j-1)r^{2/3}, 10^9j r^{2/3})$, then the parallelograms $R_{0,r}^{\wt j r^{2/3}}$ and $R_{\mathbf v - (r,r), \mathbf v}^{\wt j r^{2/3}}$ are disjoint, 
\begin{align}
& \mathbbm{1}_{H_1}\mathbbm{1}_{B_j} \mathbb{P}\Big( \log Z^{\textup{touch}, R_{0,r}^{\wt j r^{2/3}}}_{\mathcal{L}_0, n}- Z_{\mathcal{L}_0, n}^{\textup{out}, R_{0,r}^{\wt j r^{2/3}}} > 2t r^{1/3}\,\Big|\,\mathcal{F}_r \Big)\nonumber\\
& \leq \mathbbm{1}_{H_1}\mathbbm{1}_{B_j} \mathbb{P}\Big(\Big[ \max_{\mathbf{u} \in \mathcal{L}_r^{4\wt jr^{2/3}}} \log {Z}_{\mathcal{L}_0, \mathbf{u}} + \max_{\mathbf{v} \in \mathcal{L}_r^{4\wt jr^{2/3}}} \log Z_{\mathbf v, n}  \Big] \\
& \qquad \qquad \qquad \qquad \qquad \qquad - \Big[  Z_{\mathcal{L}_0, \mathbf{u^{\textup{max}}_r}}^{\textup{in}, R_{\mathbf u^{\textup{max}}_r - (r,r), \mathbf u^{\textup{max}}_r}^{\wt j r^{2/3}}} + \log Z_{\mathbf u^{\textup{max}}_r, n} \Big]> t r^{1/3}\,\Big|\,\mathcal{F}_r \Big). \label{onH}
\end{align}
Now, because we are on the event $B_j$, it holds that 
$$\max_{\mathbf{v} \in \mathcal{L}_r^{4\wt jr^{2/3}}} \log Z_{\mathbf v, n} < \log Z_{\mathbf u^{\textup{max}}_r, n} -  \wt j^{1/2} r^{1/3},$$
and plugging this into \eqref{onH} and continue the estimate,
\begin{align*}\eqref{onH} &\leq \mathbbm{1}_{H_1}\mathbbm{1}_{B_j} \mathbb{P}\Big( \max_{\mathbf{u} \in \mathcal{L}_r^{4\wt jr^{2/3}}} \log {Z}_{\mathcal{L}_0, \mathbf{u}} - Z_{\mathcal{L}_0, \mathbf{u^{\textup{max}}_r}}^{\textup{in}, R_{\mathbf u^{\textup{max}}_r - (r,r), \mathbf u^{\textup{max}}_r}^{\wt j r^{2/3}}}> (\wt  j^{1/2} + t) r^{1/3}\,\Big|\,\mathcal{F}_r \Big)\\
& \leq \mathbb{P}\Big( \max_{\mathbf{u} \in \mathcal{L}_r^{4\wt jr^{2/3}}} \log {Z}_{\mathcal{L}_0, \mathbf{u}} - \Lambda_r> \tfrac{1}{2}(\wt j^{1/2} + t) r^{1/3} \Big) \\
& \qquad \qquad \qquad + \mathbb{P}\Big( Z_{\mathcal{L}_0, r}^{\textup{in}, {\wt j r^{2/3}}} - \Lambda_r <-\tfrac{1}{2}(\wt j^{1/2}+ t) r^{1/3} \Big)\\
& \leq 2e^{-C(\wt j^{1/2} + t)}\qquad \textup{ by Proposition \ref{ptl_upper} and Proposition \ref{wide_similar}.}
\end{align*}

Next, let us look at the estimate on $H^c$,
\begin{align}
& \mathbbm{1}_{H^c}\mathbbm{1}_{B_j} \mathbb{P}\Big( \log Z^{\textup{touch}, R_{0,r}^{\wt j r^{2/3}}}_{\mathcal{L}_0, n}- Z_{\mathcal{L}_0, n}^{\textup{out}, R_{0,r}^{\wt j r^{2/3}}} > 2t r^{1/3}\,\Big|\,\mathcal{F}_r \Big)\nonumber\\
& \leq \mathbbm{1}_{H^c}\mathbbm{1}_{B_j} \mathbb{P}\Big(\Big[ \log {Z}^{\textup{touch}, R_{0,r}^{\wt j r^{2/3}}}_{\mathcal{L}_0, \mathbf{p}^*} +  \log Z_{\mathbf p^*, n}  \Big] - \Big[  Z_{\mathcal{L}_0, \mathbf{u^{\textup{max}}_r}}^{\textup{in}, R_{\mathbf u^{\textup{max}}_r - (r,r), \mathbf u^{\textup{max}}_r}^{\wt j r^{2/3}}} + \log Z_{\mathbf u^{\textup{max}}_r, n} \Big]> t r^{1/3}\,\Big|\,\mathcal{F}_r \Big)\nonumber\\
& \leq  \mathbbm{1}_{H^c}\mathbbm{1}_{B_j} \mathbb{P}\Big( \max_{\mathbf u \in \mathcal{L}_r \setminus \mathcal{L}_r^{4\wt j r^{1/3}}} \log {Z}^{\textup{touch}, R_{0,r}^{\wt j r^{2/3}}}_{\mathcal{L}_0, \mathbf{u}}  -   Z_{\mathcal{L}_0, \mathbf{u^{\textup{max}}_r}}^{\textup{in}, R_{\mathbf u^{\textup{max}}_r - (r,r), \mathbf u^{\textup{max}}_r}^{\wt j r^{2/3}}} > t r^{1/3}\,\Big|\,\mathcal{F}_r \Big)\nonumber\\
& \leq \mathbb{P}\Big( \max_{\mathbf u \in \mathcal{L}_r \setminus \mathcal{L}_r^{4\wt j r^{2/3}}} \log {Z}^{\textup{touch}, R_{0,r}^{\wt j r^{2/3}}}_{\mathcal{L}_0, \mathbf{u}}  -  (\Lambda_r - \epsilon_0\wt j r^{1/3}) > \tfrac{1}{2}t r^{1/3}\Big)\label{Hc1}\\
&  \qquad \qquad  + \mathbb{P}\Big(  Z_{\mathcal{L}_0, r}^{\textup{in}, {\wt j r^{2/3}}} -    (\Lambda_r - \epsilon_0\wt j r^{1/3})< - \tfrac{1}{2} tr^{1/3} \Big)\label{Hc2}
\end{align}
where $\epsilon_0$ is a small absolute constant which we will fix below \eqref{fixe}.
Note that \eqref{Hc2} is bounded by $e^{-C(\wt j^{1/2} + t)}$ by Proposition \ref{wide_similar}, thus to complete the proof, we will upper bound \eqref{Hc1}. 

To start, we may restrict our attention to paths between 
$\mathcal{L}_0^{20\wt j r}$ and $\mathcal{L}_r^{20 \wt j r} \setminus \mathcal{L}_r^{4\wt j r^{2/3}}$ and for $ \wt j \leq 10 r^{1/3}$, because the paths are directed and all must touch the parallelogram $R_{0,r}^{\wt j r^{2/3}}$. Then, it suffices to show for 
\begin{equation}\label{short}
\mathbb{P}\Big( \max_{\mathbf u \in \mathcal{L}_r^{20\wt j r} \setminus \mathcal{L}_r^{4\wt j r^{2/3}}} \log {Z}^{\textup{touch}, R_{0,r}^{\wt j r^{2/3}}}_{\mathcal{L}_0^{20\wt j r}, \mathbf{u}}  -  \Lambda_r  > (- \epsilon_0\wt j + \tfrac{1}{2}t) r^{1/3}\Big) \leq e^{-C(\wt j + t)}.
\end{equation} 

For values of $\wt j$ less than any absolute constant (independent of $r$ and $n$), this can be upper bounded by Proposition \ref{ptl_upper} and changing the constant $C$ on the right side of \eqref{short}. Thus from now on, we may assume that $\wt j$ is large.

We will make a decomposition of the segments
$\mathcal{L}_0^{20\wt jr}$ and $\mathcal{L}_r^{20\wt j r} \setminus \mathcal{L}_r^{4\wt j r^{2/3}}$. For odd integers $u$ and $v$ satisfy  $|u|\leq 20r^{1/3}$ and $5\leq |v|\leq 20r^{1/3}$,  let us define 
$$I_u = \mathcal{L}_{\olsi{u\wt j r^{2/3}}}^{\wt j r^{2/3}}\qquad  \textup{ and } \qquad  J_v = (r,r) +  \mathcal{L}_{\olsi{v\wt j r^{2/3}}}^{\wt j r^{2/3}}.$$
We will start by controlling the free energy between each pair of $I_u$ and $J_v$, showing that
\begin{equation}\label{uvest}
\mathbb{P}\Big(\log {Z}^{\textup{touch}, R_{0,r}^{\wt j r^{2/3}}}_{I_u, J_v}  -  \Lambda_r  > (- \epsilon_0 \wt j + \tfrac{1}{2}t) r^{1/3}\Big) \leq e^{-C(\wt j + t) +(|u|+|v|)^{1/10}}.
\end{equation}

First, if $u$ and $v$ have different signs, then the paths from $I_u$ to $J_v$ must have a large transversal fluctuation of order at least $\tfrac{1}{10}(|u|+|v|)\wt jr^{2/3}$. Recall the constant $c^*$ from Proposition \ref{trans_fluc_loss}, the expression $(-\epsilon_0 \wt j+\tfrac12t)r^{1/3}$ from \eqref{uvest} can be written as 
\begin{equation}\label{fixe}
(-\epsilon_0 \wt j^{1/2}+\tfrac12t)r^{1/3} = \big(-\tfrac{c^*}{10}(|u| +|v|)\wt j + \big(\tfrac{c^*}{10}(|u| +|v|) - \epsilon_0\big)\wt j +\tfrac12t\big)r^{1/3}
\end{equation}
and we can fix $\epsilon_0$ sufficiently small so that $\tfrac{c^*}{10}(|u| +|v|)\wt j - \epsilon_0> 0$.
Then, \eqref{uvest} holds due to Proposition \ref{trans_fluc_loss}.

Now, we look at the case when $u$ and $v$ are both positive, and we will split the estimate into two cases:
\begin{enumerate}[(1)]
\item When $|u-v| \geq \max\{3, \tfrac{1}{10}(u+v)\}$, the paths between them must have a transversal fluctuation of order at least $\tfrac{1}{10}(u+v)\wt jr^{2/3}$. Then again, \eqref{uvest} follows from Proposition \ref{trans_fluc_loss}. 
\item When $|u-v| \leq \max\{2, \tfrac{1}{10}(u+v)\}$, note we must have $u \geq 3$ since $v \geq 5$ by definition. Because the paths between $I_u$ and $J_v$ must touch $R_{0,r}^{\wt jr^{2/3}}$, they again will have a transversal fluctuation of order at least $\tfrac{1}{10}(u+v)\wt jr^{2/3}$. In this case, \eqref{uvest} follows from Proposition \ref{trans_fluc_loss3}. In our  application of Proposition \ref{trans_fluc_loss3}, let $q = \max\{2, \frac{2}{10}(u+v)\wt j\}$ and we are bounding the probability
$$\mathbb{P}\Big(\log {Z}^{\textup{exit}, {q r^{2/3}}}_{\mathcal{L}_{\olsi{u\wt jr^{2/3}}}^{2qr^{2/3}}, (r,r)+\mathcal{L}_{\olsi {u\wt jr^{2/3}}}^{2qr^{2/3}}}  -  \Lambda_r  > (-  c^*qr^{1/3} + (c^*q - \epsilon_0\wt j) + \tfrac{1}{2}t) r^{1/3}\Big)$$
which upper bounds \eqref{uvest}.

\end{enumerate}
Once we have shown \eqref{uvest}, by a union bound, 
\begin{align*}
\eqref{short} &\leq \sum_{u,v} \mathbb{P}\Big(\log {Z}^{\textup{touch}, R_{0,r}^{\wt j r^{2/3}}}_{I_u, J_v}  -  \Lambda_r  > (- \epsilon_0\wt  j + \tfrac{1}{2}t) r^{1/3} - 100(\log r +\log \wt j)\Big)\\
&\leq  \sum_{u, v} e^{-C(\wt  j + t) +(|u|+|v|)^{1/10}}
 \leq e^{-C(\wt j^{1/2} + t)}.
\end{align*}
This finishes the argument on the event $H^c$ and completes the proof of the proposition. 

\end{proof}

\subsection{Proof of Proposition \ref{prop_A} and Proposition \ref{prop_AA}}

We will start with two results that upper bounds the right and left tails of $\log Z_{\mathcal{L}_0, n} - \log Z_{\mathcal{L}_r, n} - \Lambda_r$, given by Proposition \ref{replace}  and Proposition  \ref{replace1} respectively. We will work with a slightly less optimal order $r^{1/3} \log(n/r)$, instead of the optimal order $r^{1/3}$.

\begin{proposition}\label{replace}
There exist constants $r_0, t_0$ such that for each $r\geq r_0$ and $t\geq t_0$, there exists an $\mathcal{F}_r$-measurable event $G_t$ with $\mathbb{P}(G_t)\geq 1-e^{-\sqrt t}$ such that for each $s \geq t$, the following holds,
\begin{equation}\label{341}
\mathbbm{1}_{G_t} \mathbb{P}\Big(\log Z_{\mathcal{L}_0, n} - \log Z_{\mathcal{L}_r, n} - \Lambda_r \geq sr^{1/3}\log(n/r) \,\Big|\, \mathcal{F}_r\Big) \leq e^{-\sqrt s}.
\end{equation}
Consequently, by monotonicity, for each $j = 1, \dots, j_0$, 
\begin{equation}\label{342}
\mathbbm{1}_{G_t} \mathbb{P}\Big(\log Z_{\mathcal{L}_0, n}^{\textup{out}, R_{0,r}^{\wt j r^{2/3}}} - \log Z_{\mathcal{L}_r, n} - \Lambda_r \geq sr^{1/3}\log(n/r) \,\Big|\, \mathcal{F}_r\Big) \leq e^{-\sqrt s}.
\end{equation}
\end{proposition}

\begin{proof}

To start, note that \eqref{341} implies \eqref{342} as $\log Z_{\mathcal{L}_0, n}^{\textup{out}, R_{0,r}^{\wt j r^{2/3}}} \leq \log Z_{\mathcal{L}_0, n}$. Thus, we will just need to prove \eqref{341}. And to show \eqref{341}, it suffices for us to establish the estimate by replacing $\log Z_{\mathcal{L}_0, n}$ with 
$$\max_{\mathbf{p} \in \mathcal{L}_r}  \Big\{\log Z_{\mathcal{L}_0, \mathbf p} + \log Z_{\mathbf p, n}\Big\},$$
since $\log Z_{\mathcal{L}_0, n}\leq 
\max_{\mathbf{p} \in \mathcal{L}_r}  \Big\{\log Z_{\mathcal{L}_0, \mathbf p} + \log Z_{\mathbf p, n}\Big\} + 2\log n$, and $\log n \leq 10 r^{1/3}\log(n/r)$.

Let $\mathbf p^*$ denote the unique maximizer above, and we will split the estimate into two cases, according to the events $H = \{\mathbf p^* \in \mathcal{L}_r^{tn^{2/3}}\}$ and $H^c$. 
Starting with the event $H$, 
\begin{align*}
&\mathbb{P}\Big(\Big\{\log Z_{\mathcal{L}_0, \mathbf p^*} + \log Z_{\mathbf p^*, n}- \log Z_{\mathcal{L}_r, n} - \Lambda_r \geq sr^{1/3}\log(n/r)\Big\} \bigcap H \,\Big|\, \mathcal{F}_r \Big) \\
&\leq \mathbb{P}\Big(\max_{\mathbf u \in \mathcal{L}_r^{tn^{2/3}}} \log Z_{\mathcal{L}_0, \mathbf u} - \Lambda_r \geq sr^{1/3}\log(n/r) \Big)
 \leq e^{-Cs\log(n/r)} \leq e^{-\sqrt s}
\end{align*}
where the second last inequality follows from Proposition \ref{ptl_upper}. In its application, we note that the variable $h$ appearing in Proposition \ref{ptl_upper} is set to be  $t(n/r)^{2/3}$ and the upper bound assumption on $h$ in Proposition \ref{ptl_upper} is satisfied provided that $t_0$ is fixed sufficiently large. This finishes the argument on the event $H$. 

On the event $H^c$, we start by decomposing the line segment $\mathcal{L}_r \setminus \mathcal{L}_{r}^{tn^{2/3}}$ into smaller segments of length $2n^{2/3}$. For integers $u$  with $t\leq |u|\leq n^{1/3}$, let 
$$I_u = \mathcal{L}_{(r,r) + \olsi{u n^{2/3}}}^{n^{2/3}}.$$ 
Then, for a small fixed constant $c_1$, let us define the $\mathcal{F}_r$-measurable event $G_t$ to be 
\begin{equation}\label{def_A}
G_t = \bigcap_{t\leq |u| \leq n^{1/3}}  \Big\{\log Z_{I_u, n} - \log Z_{\mathcal{L}_r, n} \leq -c_1 u^2 n^{1/3} \Big\}.
\end{equation}
For $c_1$ sufficiently small, by Proposition \ref{low_ub} and Proposition \ref{trans_fluc_loss},
\begin{align*}
 \mathbb{P}(G^c_t) &\leq \sum_{t\leq |u|\leq n^{1/3}}\mathbb{P}\Big( \log Z_{I_u, n} - \log Z_{\mathcal{L}_r, n} \geq -c_1 u^2 n^{1/3}\Big)\\
& \leq \sum_{t\leq |u|\leq n^{1/3}} \Big[\mathbb{P}\Big( \log Z_{I_u, n} - \Lambda_{n-r}\geq -\tfrac{1}{2}2c_1 u^2 n^{1/3}\Big) + \mathbb{P}\Big(  \log Z_{\mathcal{L}_r, n } - \Lambda_{n-r} \leq -\tfrac{1}{2}c_1 u^2 n^{1/3}\Big)\Big]\\
& \leq \sum_{t\leq |u|\leq n^{1/3}} e^{-Cu^2} \leq e^{-\sqrt t}.
\end{align*}
Now, on this event $G_t$, the following inequality holds 
\begin{align*}
&\mathbbm{1}_{G_t}\mathbb{P}\Big(\Big\{\log Z_{\mathcal{L}_0, \mathbf p^*} + \log Z_{\mathbf p^*, n}- \log Z_{\mathcal{L}_r, n} - \Lambda_r \geq sr^{1/3}\log(n/r)\Big\} \bigcap H^c \, \Big| \, \mathcal{F}_r \Big) \\
&\leq \sum_{t\leq |u|\leq n^{1/3}}\mathbbm{1}_{G_t}\mathbb{P}\Big(\log Z_{\mathcal{L}_0, I_u} + \log Z_{I_u, n} - \log Z_{\mathcal{L}_r, n}- \Lambda_r \geq sr^{1/3}\log(n/r) \, \Big| \, \mathcal{F}_r\Big)\\
&\leq \sum_{t \leq |u|\leq n^{1/3}} \mathbb{P}\Big(\log Z_{\mathcal{L}_0, I_u} - \Lambda_r \geq s r^{1/3}\log(n/r) + c_1u^2 \frac{n^{1/3}}{r^{1/3}} r^{1/3} \Big) \\
& \leq \sum_{|u|\geq t} e^{-C(s+u^2(n/r)^{1/3})}  \leq e^{-\sqrt s}
\end{align*}
where the first inequality on the last line again follows from Proposition \ref{ptl_upper}.
With this, we have finished the proof of the proposition. 
\end{proof}

In addition to the upper tail estimate, we will now present the lower tail estimates. 
\begin{proposition}\label{replace1}
There exist constants $r_0, t_0$ such that for each $r\geq r_0$ and $t\geq t_0$, the following holds,
$$\mathbb{P}\Big(\log Z_{\mathcal{L}_0, n} - \log Z_{\mathcal{L}_r, n} - \Lambda_r \leq - tr^{1/3}\log(n/r)\,\Big|\, \mathcal{F}_r\Big) \leq e^{-\sqrt t}.$$
In addition, for each $j=1,\dots, j_0$, it holds that 
$$\mathbbm{1}_{B_j}\mathbb{P}\Big(\log Z_{\mathcal{L}_0, n}^{\textup{out}, R_{0,r}^{\wt j r^{2/3}}} - \log Z_{\mathcal{L}_r, n} - \Lambda_r \leq -tr^{1/3}\log(n/r)\,\Big|\, \mathcal{F}_r\Big) \leq e^{-\sqrt t}.$$
\end{proposition}
\begin{proof}
For the first estimate, we may replace $\log Z_{\mathcal{L}_r, n}$ by $\max_{\mathbf u \in \mathcal{L}_r} \log Z_{\mathbf u, n}$, since $\log Z_{\mathcal{L}_r, n} \leq \max_{\mathbf u \in \mathcal{L}_r} \log Z_{\mathbf u, n} + 2 \log n$ and $\log n \leq 10 r^{1/3}\log(n/r)$ provided that $r$ and $n$ are large.

Recall that $\mathbf{u}^{\textup{max}}_r$ is the unique maximizer of $\max_{\mathbf u \in \mathcal{L}_r} \log Z_{\mathbf u, n}$ which is $\mathcal{F}_r$-measurable, then it holds that 
\begin{align*}
&\mathbb{P}\Big(\log Z_{\mathcal{L}_0, n} - \max_{\mathbf u \in \mathcal{L}_r} \log Z_{\mathbf u, n} - \Lambda_r \leq -tr^{1/3}\log(n/r)\,\Big|\, \mathcal{F}_r\Big)\\
&\leq \mathbb{P}\Big(\log Z_{\mathcal{L}_0, \mathbf{u}^{\textup{max}}_r} + \log Z_{\mathbf{u}^{\textup{max}}_r, n} - \max_{\mathbf u \in \mathcal{L}_r} \log Z_{\mathbf u, n} - \Lambda_r \leq -tr^{1/3}\log(n/r)\,\Big|\, \mathcal{F}_r\Big)\\
& = \mathbb{P}\Big(\log Z_{\mathcal{L}_0, r} - \Lambda_r \leq -tr^{1/3}\log(n/r)\Big)
 \leq e^{-\sqrt t} \qquad \text{ by Proposition \ref{low_ub}.}
\end{align*}
The second estimate is similar. Note that on the event $B_j$, $\mathbf{u}^{\textup{max}}_r$ contained inside $\mathcal{L}_r \setminus \mathcal{L}_r^{10\wt jr^{2/3}}$, then 
\begin{align*}
&\mathbbm{1}_{B_j}\mathbb{P}\Big(\log Z_{\mathcal{L}_0, n}^{\textup{out}, R_{0,r}^{\wt j r^{2/3}}} - \max_{\mathbf u \in \mathcal{L}_r} \log Z_{\mathbf u, n} - \Lambda_r \leq -tr^{1/3}\log(n/r)\,\Big|\, \mathcal{F}_r\Big)\\
&\leq \mathbbm{1}_{B_j} \mathbb{P}\Big(\log Z_{\mathcal{L}_0, \mathbf{u}^{\textup{max}}_r}^{\textup{in}, R_{\mathbf{u}^{\textup{max}}_r - (r,r),\mathbf{u}^{\textup{max}}_r}^{\wt j r^{2/3}}} +  \log Z_{\mathbf{u}^{\textup{max}}_r, n}  -  \max_{\mathbf u \in \mathcal{L}_r} \log Z_{\mathbf u, n} - \Lambda_r \leq -tr^{1/3}\log(n/r)\,\Big|\, \mathcal{F}_r\Big)\\
& \leq \mathbb{P}\Big(\log Z_{\mathcal{L}_0, r}^{\textup{in}, {\wt j r^{2/3}}} - \Lambda_r \leq -tr^{1/3}\log(n/r)\Big)\\
& \leq \mathbb{P}\Big(\log Z_{\mathcal{L}_0, r}^{\textup{in}, {r^{2/3}}} - \Lambda_r \leq -tr^{1/3}\log(n/r)\Big) \leq e^{-\sqrt t} \qquad \textup{ by Proposition \ref{wide_similar}.}
\end{align*}
This finishes the proof of the proposition.
\end{proof}

Finally, we address Proposition \ref{prop_A} and Proposition \ref{prop_AA}, starting with Proposition \ref{prop_AA}.

\begin{proof}[Proof of Proposition \ref{prop_AA}]
Let $G_t$ be defined as in \eqref{def_A}. 
Utilizing Proposition \ref{replace} and Proposition \ref{replace1}, we obtain that for $t\geq t_0$ and $s\geq t$, it holds that 
\begin{equation}\label{replace2} \mathbbm{1}_{G_t}\mathbb{P}\Big(\Big|\log Z_{\mathcal{L}_0, n} - \log Z_{\mathcal{L}_r, n} -  \Lambda_r\Big|  \geq sr^{1/3}\log(n/r)  \, \Big| \, \mathcal{F}_r \Big) \leq e^{-\sqrt s}.
\end{equation}
Now with this, we will give an upper bound for the following  conditional expectation 
\begin{align*}
& \mathbbm{1}_{G_t}\mathbb{E}\Big[\Big|\log Z_{\mathcal{L}_0, n} - \log Z_{\mathcal{L}_r, n} -  \Lambda_r \Big|\,\Big| \,\mathcal{F}_r \Big]\\
& =  \mathbbm{1}_{G_t}\int_{[0,\infty)} \mathbb{P}\Big(\Big|\log Z_{\mathcal{L}_0, n} - \log Z_{\mathcal{L}_r, n} -  \Lambda_r\Big|  \geq x  \, \Big| \, \mathcal{F}_r \Big) dx \\
&  \leq r^{1/3} \log(n/r)\int_{[0,\infty)} \mathbbm{1}_{G_t}\mathbb{P}\Big(\Big|\log Z_{\mathcal{L}_0, n} - \log Z_{\mathcal{L}_r, n} -  \Lambda_r\Big|  \geq sr^{1/3}\log(n/r)  \, \Big| \, \mathcal{F}_r \Big) ds\\
& \leq  r^{1/3} \log(n/r)\Big[t + \int_{[t,\infty)} \mathbbm{1}_{G_t}  \mathbb{P}\Big(\Big|\log Z_{\mathcal{L}_0, n} - \log Z_{\mathcal{L}_r, n} -  \Lambda_r\Big|  \geq sr^{1/3}\log(n/r)  \, \Big| \, \mathcal{F}_r \Big) ds \\
& \leq  r^{1/3} \log(n/r)\Big[t + \int_{[t,\infty)} e^{-\sqrt s} ds\Big] \qquad  \text{ by \eqref{replace2}}\\
& \leq 2tr^{1/3}\log (n/r).
\end{align*}
Then using this, the following estimate would conclude our proof of the proposition 
\begin{align}
&\mathbbm{1}_{G_t}\mathbb{P}\Big(\Big|\log Z_{\mathcal{L}_0, n} - \mathbb{E}\Big[\log Z_{\mathcal{L}_0, n} \,\Big| \,\mathcal{F}_r \Big]\Big| \geq 4tr^{1/3}\log(n/r) \, \Big| \, \mathcal{F}_r \Big)\nonumber\\
& \leq \mathbbm{1}_{G_t}\mathbb{P}\Big(\Big|\log Z_{\mathcal{L}_0, n} - \log Z_{\mathcal{L}_r, n} -  \Lambda_r\Big| \nonumber\\
& \qquad \qquad \qquad \qquad \qquad + \Big| \mathbb{E}\Big[\log Z_{\mathcal{L}_0, n} - \log Z_{\mathcal{L}_r, n} -  \Lambda_r \,\Big| \,\mathcal{F}_r \Big]\Big| \geq 4tr^{1/3}\log(n/r)\, \Big| \, \mathcal{F}_r \Big)\nonumber\\
& \leq \mathbbm{1}_{G_t}\mathbb{P}\Big(\Big|\log Z_{\mathcal{L}_0, n} - \log Z_{\mathcal{L}_r, n} -  \Lambda_r\Big| + 2tr^{1/3}\log(n/r) \geq 4tr^{1/3}\log(n/r) \, \Big| \, \mathcal{F}_r \Big) \nonumber\\
& = \mathbbm{1}_{G_t}\mathbb{P}\Big(\Big|\log Z_{\mathcal{L}_0, n} - \log Z_{\mathcal{L}_r, n} -  \Lambda_r\Big| \geq 2t r^{1/3}\log(n/r) \, \Big| \, \mathcal{F}_r \Big) \nonumber\\
& \leq e^{\sqrt t}\qquad \text{ by \eqref{replace2}}. \nonumber
\end{align}
This finishes the proof of the proposition.
\end{proof}
In the argument above, the only input required is \eqref{replace2}. From Proposition \ref{replace} and Proposition \ref{replace1}, we can also derive that
\begin{equation}\mathbbm{1}_{G_t\cap B_j}\mathbb{P}\label{replace4}
\Big(\Big|\log Z_{\mathcal{L}_0, n}^{\textup{out}, R_{0,r}^{\wt j r^{2/3}}} - \log Z_{\mathcal{L}_r, n} - \Lambda_r\Big| \geq sr^{1/3}\log(n/r) \, \Big| \, \mathcal{F}_r \Big) \leq e^{-\sqrt{s}}.
\end{equation}
By replacing the estimate \eqref{replace2} in the proof of Proposition \ref{replace1} with \eqref{replace4}, the same calculation leads directly to Proposition \ref{replace}. Therefore, we omit this repeated argument.

\bibliographystyle{amsplain}
\bibliography{time}

\end{document}